\documentclass[12pt,a4paper]{amsart}
\usepackage[T1]{fontenc}
\usepackage[top=35mm, bottom=35mm, left=30mm, right=30mm]{geometry}
\usepackage{amssymb,amsthm}

\usepackage[looser]{newpxtext}
\usepackage[smallerops]{newpxmath}
\usepackage{verbatim}
\usepackage[all]{xy}

\usepackage[hidelinks,backref]{hyperref}

\usepackage[nobysame]{amsrefs}

\usepackage{fancyhdr}
\makeatletter
\def\@@and{\MakeLowercase{and}}
\makeatother

\newtheorem{thm}{Theorem}[section]
\newtheorem{cor}[thm]{Corollary}

\newtheorem{lem}[thm]{Lemma}
\newtheorem{prop}[thm]{Proposition}

\theoremstyle{definition}

\newtheorem{rem}[thm]{Remark}
\numberwithin{equation}{section}

\newcommand{\eps}{\varepsilon}
\newcommand{\bbn}{\mathbb{N}} 
\newcommand{\ip}{\mathrm{IP}}
\newcommand{\RP}{{\bf RP}} 
 
\DeclareMathOperator{\eq}{Eq}
\DeclareMathOperator{\diam}{diam}

\makeatletter
\@namedef{subjclassname@2020}{%
	\textup{2020} Mathematics Subject Classification}
\makeatother

\title[C\MakeLowercase{haracterizations of distality via weak equicontinuity}] 
{C\MakeLowercase{haracterizations of distality via weak equicontinuity}}

\author[J. L\MakeLowercase{i}]{J\MakeLowercase{ian} Li}
\address[J. Li]{Department of Mathematics,
	Shantou University, Shantou, 515063, Guangdong, China}
\email{lijian09@mail.ustc.edu.cn}
\urladdr{https://orcid.org/0000-0002-8724-3050}

\author[Y. Y\MakeLowercase{ang} ]{Y\MakeLowercase{ini}  Yang}
\address[Y. Yang]{School of Mathematical Sciences, Xiamen University, Xiamen, 361005, Fujian, China}
\email{ynyangchs@foxmail.com}
\urladdr{https://orcid.org/0000-0001-6564-2213}

	

\thanks{J. Li was supported in part by NFS of China (Grant Nos. 12222110 and 12171298). Y. Yang is the corresponding author.}

\begin{document}
	
\begin{abstract}
For an infinite discrete group $G$ acting on a compact metric space $X$,
we introduce several weak versions of equicontinuity along subsets of $G$ and 
show that if a minimal system $(X,G)$ admits an invariant measure then  $(X,G)$ is distal if and only if it is pairwise IP$^*$-equicontinuous; if the product system $(X\times X,G)$ of a minimal system $(X,G)$ has a dense set of minimal points, then  $(X,G)$ is distal if and only if it is pairwise IP$^*$-equicontinuous if and only if it is pairwise central$^*$-equicontinuous; if $(X,G)$ is a minimal system with $G$ being abelian, then $(X,G)$ is a system of order $\infty$ if and only if
it is pairwise FIP$^*$-equicontinuous.
\end{abstract}

\maketitle

\section{Introduction}

By a \textit{topological dynamical system} or a \textit{$G$-action} we mean a pair $(X,G)$, where $X$ is a compact metric space
with a compatible metric $d$ and $G$ is an infinite discrete group acting continuously on $X$. When $G=\mathbb{Z}$ the action can be generated
by a homeomorphism $T$ from $X$ to itself, in this case we will simply write the system as $(X,T)$.

A topological dynamical system $(X,G)$ is called \emph{equicontinuous} if for any $\eps>0$ there is a $\delta>0$
such that for any $x,y\in X$ with $d(x,y)<\delta$ one has $d(gx,gy)<\eps$ for all $g\in G$, and \emph{distal} if for any two distinct points $x,y\in X$ one has $\inf_{g\in G}d(gx,\allowbreak gy)>0$.
It is clear that every equicontinuous system is distal.
We say that a system $(X,G)$ is \emph{minimal} if no proper closed subset of $X$ is invariant under the action of $G$.
It used to be an open question whether every minimal distal system is equicontinuous (see e.g. \cite[Question (4) in page 349]{G58}).
In \cite{AHM61}, Auslander, Hahn and Markus answered this question negatively by exhibiting the existence of analytic flows on compact nilmanifolds which are distal, minimal but not equicontinuous. 
In \cite{F61} Furstenberg studied conditions of a class of homeomorphisms on the torus to be strictly ergodic.
In fact, many homeomorphisms in this class are distal, minimal but not equicontinuous. 
A simple example is as follows: $T\colon \mathbb{T}^2\to \mathbb{T}^2$, $(x,y)\mapsto (x+\alpha,x+y)$, 
where $\mathbb{T}=\mathbb{R}/\mathbb{Z}$ and $\alpha$ is irrational.
Motivated by this two classes of distal not-equicontinuous minimal system,  Furstenberg obtained the structure of  minimal distal systems in \cite[Theorems~2 and~3]{F63}: 
every minimal system $(X,G)$ is distal if and only if  a succession (possibly transfinite) of isometric extensions starting with the one-point system.

In \cite{C63} Clay introduced two variations of equicontinuity.  
It is shown in \cite[Theorem~2]{C63} that a minimal system is syndetic equicontinuous if and only it is locally almost periodic, 
and \cite[Theorem~8]{C63} that a minimal system is weak syndetic equicontinuous if and only if it is proximal equicontinuous, 
that is, it is an proximal extension to its maximal equicontinuous factor.
It is a natural question that how to characterize distality by a weak form of equicontinuity.
In \cite{F81}, Furstenberg characterized distal points in terms of the set of return times for $\mathbb{N}$-actions, 
that is, a point is distal if and only if it is IP$^*$-recurrent if and only if  it is central$^*$-recurrent  (see \cite[Theorem~9.11 and Proposition~9.17]{F81}).
Inspired by the above results, in \cite{LY21} the authors of this paper proposed a weak form of equicontinuity to characterize distality of $\mathbb{N}$-actions. 
To be precise, a $\mathbb{N}$-system $(X,T)$  is called pairwise IP$^*$-equicontinuous if  for any $\eps>0$ there is a $\delta>0$
such that for any $x,y\in X$ with $d(x,y)<\delta$ one has $d(T^nx,T^ny)<\eps$ for all $n$ in an IP$^*$-subset of $\mathbb{N}$.
It is show in  \cite[Corollary~2]{LY21} that a minimal $\mathbb{N}$-system $(X,T)$ is
distal if and only if it is pairwise IP$^*$-equicontinuous.
One of the motivations of this article is to generalize this result to the general group actions and study related property of pairwise IP$^*$- equicontinuity. 
It should be noticed that the proof of the dichotomy result for pairwise IP$^*$-equicontinuous in \cite[Corollary~2]{LY21} 
relied on the structure theorem  developed by Ye and Yu in \cite{YY18}, and in this paper we get an  independent proof of this dichotomy result. 
On the other hand,  by using the structure theorem of minimal systems we further characterize minimal distal systems by pairwise central$^*$-equicontinuity.
The first main result of this paper is as follows.

\begin{thm}\label{thm:main-resut1}
    Let $(X,G)$ be a minimal system.  
	\begin{enumerate}
		\item If $(X,G)$ admits an invariant measure, then  $(X,G)$ is distal if and only if it is pairwise IP$^*$-equicontinuous.
        \item If  $(X\times X,G)$ has a dense set of minimal points, then $(X,G)$ is distal if and only if it is pairwise IP$^*$-equicontinuous if and only if it is pairwise central$^*$-equicontinuous.
	\end{enumerate} 
\end{thm}

It is well known that a dynamical system $(X,G)$ is distal if and only if every point in the product system $(X\times X,G)$ is minimal.
In \cite[Theorem~12.3]{F63} Furstenberg proved that for a dynamical system $(X,G)$ with $G$ being locally compact if $(X,G)$ is distal then it admits an invariant measure.
If the acting group $G$ is abelian, then every minimal system $(X,G)$ admits an invariant measure and the product system $(X\times X,G)$ has a dense set of minimal points. Thus in the case of abelian group actions we can remove those conditions in (1) and (2) of Theorem~\ref{thm:main-resut1}. It is interesting to know whether we can remove those conditions in general. 

\medskip

Nilsystems are a class of systems of the form $(X,T)$ where $X$ is a nilmanifold (a compact homogeneous space of a nilpotent Lie group ) and $T$ is a niltranslation (a translation on $X$ defined by an element). It is well known that nilsystems are distal. 
Nilsystem is important in the study of the convergence of some nonconventional ergodic averages \cite{HK05}.
In \cite{HKM10} Host, Kra and Maass introduced the regionally proximal relation $\RP^{[k]}$ of order $k$ for $k\in\mathbb{N}$, and showed  in \cite[Theorem 1.3]{HKM10} that $\RP^{[k]}$ is an invariant closed equivalence relation for a minimal distal system, 
and  in \cite[Theorem 1.2]{HKM10} that for a minimal system $\RP^{[k]}$ is trivial if and only if the
system itself is an inverse limit of $k$-step nilsystems.
In \cite[Theorem 3.5]{SY2012} Shao and Ye showed that $\RP^{[k]}$ is an invariant closed equivalence relation for any minimal system and also remarked this result can be extended to abelian group actions without difficulty.
In \cite{DSMSY13} Dong et al. introduced $\infty$-step nilsystems which are the systems with trivial $\RP^{[\infty]}:=\bigcap_{k=1}^\infty \RP^{[k]}$ and showed in \cite[Theorem 3.6]{DSMSY13}  that a minimal system is an $\infty$-step nilsystem if and only if it is an inverse limit of minimal nilsystems.
There exist some minimal distal systems which are not $\infty$-step nilsystem, see \cite[Section 5]{DSMSY13}.
Distal systems are characterized by the IP$^*$-recurrence property for $\mathbb{N}$-actions \cite{F81}. 
In \cite[Theorem 8.1.7]{HSY2016} Huang, Shao and Ye proved that 
a minimal system $(X,T)$ is $\infty$-step nilsystem if and only if every point in $X$ is FIP$^*$-recurrent. 
Bergelson and Leibman \cite{BL18} provided a characterization, in terms of recurrence properties, of nilsystems.
In general, we say that a group action system $(X,G)$ is of order $k$ if the regionally proximal relation of order $k$ is trivial, see e.g. \cite{GGY18}.
Motivated by the above results,  we have the following characterization of  systems of order $\infty$.

\begin{thm}\label{thm:main-resut2}
	Let $(X,G)$ be a  dynamical system with $G$ being abelian.
	Then $(X,G)$ is a system of order $\infty$ if and only if
	it is pairwise FIP$^*$-equicontinuous.
\end{thm}

The organization of this the paper is as follows. 
In Section 2, we give some basic definitions and related results which will be used later. 
In Section 3, we introduce pairwise IP$^*$-equicontinuity and prove a slight generalization of Theorem~\ref{thm:main-resut1}~(1).
We also study the concept of almost pairwise IP$^*$-equicontinuity and obtain a dichotomy result for minimal systems and the structure of almost pairwise IP$^*$-equicontinuous system.
In Section 4, we first recall the structure theorem for (metric) minimal systems and then prove Theorem~\ref{thm:main-resut1}~(2). 
In Section 5, we introduce pairwise FIP$^*$-equicontinuity and prove Theorem~\ref{thm:main-resut2}.
 
\section{Preliminaries}

\subsection{Topological dynamical system}
Let $X$ be a compact metric space with a compatible metric $d$ and $(G,\cdot)$ be an infinite discrete
group with the identity element $e$.
A \emph{$G$-action} on $X$ is a continuous map $\Pi\colon G\times X\to X$ satisfying
$\Pi(e,x)=x$, $\forall x\in X$
and $\Pi(g,\Pi(h,x))=\Pi(gh,x)$, $\forall x\in X$, $g,h\in G$.
We say that the triple $(X,G,\Pi)$ is a \emph{topological dynamical system}.
For convenience, we will use the pair $(X,G)$ instead of $(X,G,\Pi)$ to denote the topological dynamical system, and $gx:=\Pi(g,x)$ if the map $\Pi$ is unambiguous.
For any $n\in \bbn$, there is a natural $G$-action on the $n$-fold product space $X^n$
as $g(x_1,\dotsc,x_n)=(gx_1,\dotsc,gx_n)$ for every $(x_1,\dotsc,x_n)\in X^n$. The \emph{orbit} of a point $x\in X$ is the set
$Gx:=\{gx\colon\ g\in G\}$.

A nonempty $G$-invariant closed subset $Y\subseteq X$ defines naturally a subsystem $(Y, G)$ of $(X, G)$.
A dynamical system $(X,G)$ is called \emph{minimal} if it contains no proper subsystems.
Each point belonging to some minimal subsystem of $(X,G)$ is called a \emph{minimal point}.
By  Zorn's Lemma, every topological dynamical system has a minimal subsystem.
For two subsets $U$ and $V$ of $X$, define the \emph{return time set}  of $U$ and $V$ by 
\[
N(U,V):=\{g\in G\colon gU\cap V\neq\emptyset\}.
\]
For a point $x\in X$, we write $N(x,V)$ instead of $N(\{x\},V)$ for simplicity. 
A dynamical system $(X,G)$ is called \emph{transitive} if for any two nonempty open subsets $U$ and $V$ of $X$ one has $N(U,V)\neq\emptyset$, and \emph{weakly mixing} if the product system $(X\times X,G)$ is transitive. 
A point $x\in X$ is called \emph{transitive} if the orbit of $x$ is dense in $X$.
Recall that a subset of $X$ is \emph{residual} if it contains a dense $G_\delta$ subset of $X$.
As $X$ has a countable basis, every transitive system has a residual subset of transitive points. 

\medskip
We say that a dynamical system $(X,G)$  is \emph{equicontinuous} if for any $\varepsilon>0$, there exists $\delta>0$ such that if $d(x,y)<\delta$ then $d(gx,gy)<\varepsilon$ for all $g\in G$. 
A pair $(x_1,x_2)\in X\times X$ is said to be \emph{proximal} if $\inf_{g\in G}d(gx_1, gx_2)=0$,
and \emph{distal} if $\inf_{g\in G}d(gx_1, gx_2)>0$. 
A point $x\in X$ is \emph{distal} if for any $y\in \overline{Gx}\setminus \{x\}$, $(x,y)$ is a distal pair.
We say that a dynamical system $(X,G)$  is \emph{distal} if any two distinct points of $X$ form a distal pair.
It is easy to see that a dynamical system is distal if and only if every point is distal, and every equicontinuous system is distal. 
A dynamical system $(X,G)$ is called \emph{point-distal} if the collection of distal points is residual in $X$.

The following characterization of distal systems is a classical result, see e.g. \cite[Theorem 5.6]{A88}.

\begin{lem} \label{lem:distal-prod-pointwise-minimal}
A dynamical system $(X,G)$ is distal if and only if every point in the product system $(X\times X,G)$ is minimal.
\end{lem}

Denote by  $P(X,G)$ the proximal relation on $X$, that is the collection of all proximal pairs of $(X,G)$.
It is clear that $P(X,G)$ is a reflexive, symmetric, $G$-invariant relation, but it is in general not transitive or closed. 
We will need the following two result about the proximal relation. 

\begin{thm}[{\cite[Theorem 6.13]{A88}}]\label{thm:proximal}
Let $(X,G)$ be a dynamical system and $x\in X$. Then for any minimal subset $M$ of $\overline{Gx}$, there exists some $y\in M$ such that $(x,y)$ is proximal.
\end{thm}

\begin{lem}[{\cite[Corollary 6.11]{A88}}] 
\label{lem:proximal-closed}
	Let $(X,G)$ be a dynamical system.
	If the proximal relation $P(X,G)$ is closed, then it is an equivalence relation on $X$.
\end{lem}

Let $M(X)$ be the set of all Borel probability measures on $X$. 
We say a measure $\mu\in M(X)$ has \emph{full support} if $\mu(U)>0$ for all nonempty open subset $U$ of $X$, and  is \emph{$G$-invariant} if $\mu(A)=\mu(gA)$ for all $g\in G$ and Borel subset $A$ of $X$.
It is easy to see that every invariant measure on a minimal system has full support.

\subsection{Factor map and the distal structure relation}

Let $(X,G)$ and $(Y,G)$ be two dynamical systems. If there is a continuous surjection $\pi: X \to Y$ with $\pi\circ g = g\circ \pi$ for all $g\in G$,
then we say that $\pi$ is a \emph{factor map},
the system $(Y,G)$ is a \emph{factor} of $(X,G)$
or $(X, G)$ is an \emph{extension} of $(Y,G)$. 
If $\pi$ is a homeomorphism, then we say that $\pi$ is a \emph{conjugacy} and
dynamical systems $(X,G)$ and $(Y,G)$ are \emph{conjugate}. Conjugate dynamical systems can
be considered the same from the dynamical point of view.
Let $\pi\colon (X,G)\to (Y,G)$ be a factor map between two dynamical systems and let
\[
R_\pi=\{(x_1,x_2)\in X\times X \colon \pi(x_1)=\pi(x_2)\}.
\]
Then $R_\pi$ is a $G$-invariant closed equivalence relation on $X$ and $Y=X/R_\pi$. 
In fact, there exists an one-to-one correspondence between
the collection of factors of $(X,G)$ and the collection of $G$-invariant closed equivalence relations on $X$.

A factor map $\pi\colon (X, G) \to (Y, G)$ is called \emph{proximal} if $R_{\pi}\subset P(X,G)$; \emph{almost one-to-one} if $\{x \in X \colon \pi^{-1}(\pi(x)) \textrm{ is a singleton} \}$ is residual in $X$;
\emph{isometric} if for any $\varepsilon>0$, there exists $\delta>0$ such that if 
$x_1, x_2\in R_\pi$ with $d(x_1,x_2)<\delta$ then 
$d(gx_1,gx_2)<\varepsilon$ for all $g\in G$; \emph{weakly mixing} if $(R_\pi,G)$ is transitive. For a factor map $\pi\colon (X,G)\to (Y,G)$ between minimal systems, $\pi$ is almost one-to-one if and only if there exists a point $y\in Y$ such that $\pi^{-1}(y)$ is a singleton (see e.g.~\cite[Corollary 3.6]{Li-Yang1}).
Let $X$ and $Y$ be compact metric spaces.
A map $\pi\colon X\to Y$ is called \emph{semi-open} if for every nonempty open subset $U$ of $X$, $\pi(U)$ has a non-empty interior.
It is easy to see that a map $\pi\colon X\to Y$ is semi-open if and only if for every dense subset $A$ of $Y$, $\pi^{-1}(A)$ is dense in $X$.
Note that any factor map $\pi\colon (X,G)\to (Y,G)$ between minimal systems is semi-open (see e.g.~\cite[Theorem~1.15]{A88}).

Every topological dynamical system $(X,G)$ has a maximal distal factor $(X_d,G)$. That is, $(X_{d},G)$ is a distal factor
and every distal factor of $(X,G)$ is also a factor of $(X_{d}, G)$.
There exists a 
closed $G$-invariant equivalence relations $S_d$ on $(X,G)$,  called the \emph{distal structure relation}, such that $X/S_d=X_d$.
It is well known that $S_d$ is the smallest closed $G$-invariant
equivalence relation containing $P(X,G)$ (see e.g. \cite[Exercise 9.3]{A88}).

\subsection{IP-set, FIP-set and central set}
Let $(G,\cdot)$ be an infinite discrete group.
For a sequence $\{p_i\}_{i\in I}$ in $G$ indexed by a subset $I$ of $\mathbb{N}$, we define the \emph{finite product} of $\{p_i\}_{i\in I}$  by 
\[
FP(\{p_i\}_{i\in I}) =\biggl\{\prod_{i\in \alpha} p_i\colon  \alpha
\text{ is a non-empty finite subset of }  I \biggr\},
\]
where $\prod_{i\in \alpha} p_i$ is the product in increasing order of indices.
A subset $F$ of $G$
is called an \emph{IP-set} if
there exists an infinite sequence $\{p_i\}_{i=1}^\infty$ in $G$ such that
 $FP(\{p_i\}_{i=1}^{\infty})\subset F$.
The following result follows from the well-known Hindman theorem, see e.g.~\cite[Corollary 5.15]{HS2012}.
\begin{thm}\label{thm:hindman}
The collection of IP-sets has the Ramsey property, that is, for every IP-set $F\subset G$ and $F=F_1\cup F_2$, either $F_1$ or $F_2$ is an IP-set. 
\end{thm}

A subset $F$ of $G$ is called an \emph{IP$^{*}$-set} 
if for every IP-set $A\subset G$, $F\cap A\neq\emptyset$.  
An immediate consequence of Theorem \ref{thm:hindman} is the following property of IP$^{*}$-sets, see e.g.~\cite[Theorem 16.6]{HS2012}.
\begin{prop}\label{prop:IP*-set}\leavevmode
\begin{enumerate}
    \item The intersection of two IP$^{*}$-sets is still an IP$^{*}$-set.
    \item A subset $F$ of $G$ is an IP$^{*}$-set if and only if for every IP-set $A\subset G$, $F\cap A$ is an IP-set.
\end{enumerate}
\end{prop}

A subset $F$ of $G$ is called an \emph{FIP-set} if  for every $k\in\bbn$, there exists a finite sequence $\{p_i^{(k)}\}_{i=1}^k$ of $G$ such that $FP(\{p_i^{(k)}\}_{i=1}^k)\subset F$,
and an  \emph{FIP$^*$-set} if for every FIP-set $A\subset G$, $F\cap A$ is nonempty.
It should be noticed that a FIP$^*$-set is called an IP$^*_0$-set in \cite{BL18}.
The following result follows from \cite[Corollary~3]{GR1971} directly.  

\begin{thm}\label{thm:fip-set-Ramsey}
The collection of  FIP-sets has the Ramsey property, that is, for every FIP-set $F\subset G$ and $F=F_1\cup F_2$, either $F_1$ or $F_2$ is an FIP-set. 
\end{thm}

Similar to Proposition~\ref{prop:IP*-set}, we have the following result about FIP$^{*}$-sets.
\begin{prop}\label{prop:FIP*-set}\leavevmode
\begin{enumerate}
    \item The intersection of two FIP$^{*}$-sets is  also an FIP$^{*}$-set.
    \item A subset $F$ of $G$ is an FIP$^{*}$-set if and only if for every FIP-set $A\subset G$, $F\cap A$ is  an FIP-set.
\end{enumerate}
\end{prop}

A subset $F$ of $G$ is called a \emph{central set} if there exists a dynamical system $(X,G)$, $x,y\in X$ and a neighborhood $U$ of $y$ such that $(x,y)$ is a proximal pair,
$y$ a minimal point, and $F\supset N(x,U)$.
A subset $F$ of $G$ is called a \emph{central$^{*}$-set} 
if for every central set $A\subset G$, $F\cap A\neq\emptyset$. 

The following result is proved in \cite[Prosition 8.10]{F81} for the case of $\bbn$. In fact, the proof holds for general group actions naturally.
\begin{prop}
Every central set is an IP-set.
\end{prop}

We will need the following two results on central sets and central$^*$-sets.
\begin{thm}[{\cite[Theorem 19.27]{HS2012}}]
\label{thm:central-set-Ramsey}
The collection of  central sets has the Ramsey property, that is, for every central set $F\subset G$ and $F=F_1\cup F_2$, either $F_1$ or $F_2$ is a central set. 
\end{thm}

\begin{prop}[{\cite[Lemma 15.4]{HS2012}}]\label{prop:central*-set}\leavevmode
\begin{enumerate}
    \item The intersection of two central$^{*}$-sets is still a central$^{*}$-set.
    \item A subset $F$ of $G$ is a central$^{*}$-set if and only if for every  central set $A\subset G$, $F\cap A$ is a central set.
\end{enumerate}
\end{prop}

Let $(X,G)$ be a dynamical system.
A point $x\in X$ is called \emph{IP$^*$-recurrent} (resp. \emph{FIP$^*$-recurrent}, \emph{central recurrent}, \emph{central$^*$-recurrent}) if for any neighborhood $U$ of $x$, $N(x,U)$ is an
IP$^*$-set (resp. an FIP$^*$-set, a central set, a central$^*$-set).
We have the following characterization of distal points, see \cite[Theorem 9.11 and Proposition 9.17]{F81} for $\bbn$-actions
and \cite[Corollary 5.30]{EEN} for general group actions.

\begin{thm}\label{thm:distal-ip-star-rec}
	Let $(X,G)$ be a dynamical system and $x\in X$.
	Then $x$ is distal if and only if it is IP$^*$-recurrent  if and only if it is central$^*$-recurrent.
\end{thm}

We will use the following characterization of central sets, see~\cite[Propositon~5.8]{L12} for $\bbn$-actions, but the proof holds for general groups without strain.

\begin{prop}\label{prop:central-chara}
A subset $F$ of $G$ is central if and only if for every dynamical system $(X,G)$ and $x\in X$
there exists a minimal point $y$ in $\overline{Fx}$ such that $(x,y)$ is proximal, 
where $Fx=\{gx\colon g\in F\}$.
\end{prop}

\section{\texorpdfstring{Pairwise IP$^*$-equicontinuity and distality}{Pairwise IP*-equicontinuity and distality}}

In this section we first introduce pairwise IP$^*$-equicontinuity and give a characterization of distality via pairwise IP$^*$-equicontinuity.
Then we consider the local version of pairwise IP$^*$-equicontinuity, named almost pairwise IP$^*$-equicontinuity, give the  dichotomy result for minimal systems and characterize the distal structure relation. 
It is worth mentioning that although some proofs in subsection~\ref{section:ip-star} are similar to the ones in~\cite{LY21}, we contribute new ideas to this section in subsection~\ref{subsection:almost-ip-star}.

\subsection{\texorpdfstring{Pairwise IP$^*$-equicontinuity}{Pairwise IP*-equicontinuity}}\label{section:ip-star}
Let $(X,G)$ be a dynamical system. 
A point $x$ in $X$ is called \emph{pairwise IP$^*$-equicontinuous} 
if for any $\eps>0$ there exists a neighbourhood $U$ of $x$ such that for any $y,z\in U$, $\{g\in G\colon d(gy, gz)<\eps\}$ is an IP$^*$-set. 
The dynamical system $(X,G)$ is called \emph{pairwise IP$^*$-equicontinuous} if every point in $X$ is pairwise IP$^*$-equicontinuous.
By the compactness of $X$ it is easy to see that $(X,G)$ is pairwise IP$^*$-equicontinuous if and only if for any $\eps>0$ there exists $\delta>0$ 
such that for any $x,y \in X$ with $ d(x,y)<\delta$, $\{g\in G\colon d(gx, gy)<\eps\}$ is an IP$^*$-set. 

First we have the following elementary results.

\begin{lem} \label{lem:distal-system-IP-star-eq}
	If $(X,G)$ be a distal system, then it is pairwise IP$^*$-equicontinuous.
\end{lem}
\begin{proof}
	Fix a point $x\in X$ and $\eps>0$.
	Let $U=B(x,\frac{\eps}{3})$ and $y,z\in U$.
	As $y,z$ is distal, $N(y,U)$ and $N(z,U)$ are IP$^*$-sets.
	By Proposition~\ref{prop:IP*-set}, $F:=N(y,U)\cap N(z,U)$ is also an IP$^*$-set.
    Then for any $g\in F$, $d(gy,gz)\leq \diam (U)< \eps$.
    Therefore, $x$ is pairwise IP$^*$-equicontinuous.
	As this holds for arbitrary $x$,
	$(X,G)$ is pairwise IP$^*$-equicontinuous.
\end{proof}

\begin{lem}\label{lem:IP*-eq-limit-distal}
Let $(X,G)$ be a dynamical system. 
If $x\in X$ is pairwise IP$^*$-equicontinuous and a limit point of distal points, then $x$ is a distal point.
\end{lem}
\begin{proof}
For any $\eps>0$ there exists a neighbourhood $U$ of $x$ such that for any $y\in U$, $\{g\in G\colon d(gx, gy)<\eps\}$ is an IP$^*$-set. 
Pick a distal point $z\in U\cap B(x,\eps)$. By Theorem~\ref{thm:distal-ip-star-rec}, $N(z,U\cap B(x,\eps))$ is an IP$^*$-set.
By Proposition~\ref{prop:IP*-set}, 
$F:=\{g\in G\colon d(gx, gz)<\eps\}\cap  N(z,U\cap B(x,\eps))$ is also an IP$^*$-set.  By the triangular inequality property of the metric $d$,
$F\subset \{g\in G\colon d(gx, x)<2\eps\}$. This implies that $x$ is IP$^*$-recurrent and then $x$ is a distal point.
\end{proof}

Following~\cite{Akin1997}, we say that a dynamical system $(X,G)$ is IP$^*$-central\footnote{It should be noticed that here ``central'' means the property of the return time set $N(U,U)$, which is quite different from the central set.} 
if for every nonempty open subset $U$ of $X$, the return time set $N(U,U)$ is an IP$^*$-set.
By Theorem~\ref{thm:distal-ip-star-rec}, we known that if a dynamical system has a dense set of distal points, then it is IP$^*$-central. 

\begin{lem}\label{lem:supp-IP=subset}
	A dynamical system $(X,G)$ is IP$^*$-central if and only if  for any IP-set $F$ in $G$ and nonempty open subset $U$ of $X$, there exist
	an IP-subset $F^{\prime}$ of $F$ and a point $z\in U$ such that $F^{\prime}\subset N(z,U)$.
\end{lem}

\begin{proof}
    ($\Leftarrow$)
    Fix a nonempty open subset $U$ of $X$.  For any IP-set $F$ in $G$, there exist an IP-subset $F^{\prime}$ of $F$ and a point $z\in U$ such that $F^{\prime}\subset N(z,U)$.
    It is clear that $N(z,U)\subset N(U,U)$. 
    Then $F'\subset F\cap N(z,U)\subset F\cap N(U,U)$.
    Therefore, $N(U,U)$ is an IP$^*$-set and $(X,G)$ is IP$^*$-central.
    
	($\Rightarrow$) Assume that  $(X,G)$ is IP$^*$-central. Fix an IP-set $F$ in $G$ and a nonempty open subset $U$ of $X$.
	Pick a sequence $\{p_i\}_{i=1}^{\infty}$ in $G$ such that $FP\{p_i\}_{i=1}^{\infty} \subset F$.
	Take a nonempty open subset $V_0$ of $X$ such that $\overline{V_0}\subset U$.
	Since $(X,G)$ is IP$^*$-central, $N(V_0,V_0)$ is an IP$^*$-set, there is a finite subset $\alpha_1$ of $\bbn$ such that $V_0\cap p_{\alpha_1}^{-1}V_0\neq\emptyset$ where $p_{\alpha_1}=\prod_{j\in\alpha_1}p_j$.
	Take a nonempty open subset $V_1$ of $X$ such that $\overline{V_1}\subset V_0\cap p_{\alpha_1}^{-1}V_0$,
	there is a finite subset $\alpha_2$ of $\bbn$ with $\min{\alpha_2}>\max{\alpha_1}$
	and
	$V_1\cap p_{\alpha_2}^{-1}V_1\neq \emptyset$.
	By induction we get a sequence $\{\alpha_i \}$ of finite subsets of $\bbn$ and a sequence of nonempty open set $\{V_i\}$ which satisfy  that for any $i\in\bbn$,
	$\min{\alpha_{i+1}}>\max{\alpha_i}$ and
	$\overline{V_{i+1}}\subset V_i\cap p_{\alpha_{i+1}}^{-1}V_i$.
	By the compactness of $X$, take a point $z\in\bigcap_{i=1}^{\infty}\overline{V_i}$ and let $q_i=p_{\alpha_i}$ for $i\in\bbn$. It is easy to verify that $FP\{q_i\}_{i=1}^{\infty}\subset N(z,U)$.
\end{proof}

\begin{lem}\label{lem:measure-IP-star-central}
If a dynamical system $(X,G)$ admits an invariant measure with full support, then it is IP$^*$-central.
\end{lem}
\begin{proof}
 Let $\mu$ be the invariant measure of $(X,G)$ with full support. 
 Fix an IP-set $F$ in $G$ and pick a sequence $\{p_i\}_{i=1}^{\infty}$ in $G$ such that $FP\{p_i\}_{i=1}^{\infty} \subset F$. 
 For any nonempty open subset $U$ of $X$, one has $\mu(U)>0$.
 For any $n\in\bbn$, let $q_n=\prod_{j\in\{1,2,\dots,n\}}p_j$.
 As $\mu$ is $G$-invariant, $\mu(q_n^{-1}(U))=\mu(U)$ for any $n\in\bbn$.
 Note that $\mu(X)=1$. So there exist two positive integers $n_1< n_2$ such that $\mu\bigl(q_{n_1}^{-1}(U) \cap q_{n_2}^{-1}(U) \bigr)>0$. Then $U\cap (\prod_{j\in\{n_1+1,n_1+2,\dots,n_2\}}p_j)^{-1}(U) \neq\emptyset$, and $N(U,U)$ is an IP$^*$-set. This implies that $(X,G)$ is IP$^*$-central.
\end{proof}

\begin{lem}\label{lem:IP-star-eq-point-distal}
	Let $(X,G)$ be an IP$^*$-central system.
	If $x\in X$ is pairwise IP$^*$-equicontinuous, then $x$ is distal.
\end{lem}
\begin{proof}
	Assume that $x$ is not distal. By Theorem~\ref{thm:distal-ip-star-rec} it is not IP$^*$-recurrent, 
    that is, there exists a $\delta>0$ such that $N(x,X\setminus B(x,2\delta))$ is an IP-set.
	As $x$ is pairwise  IP$^*$-equicontinuous, there exists a neighborhood $U$ of $x$
	such that for any $z\in U$,
	$\{g\in G\colon d(g x,g z)<\delta\}$ is an IP$^*$-set.
	However, as $(X,G)$ is IP$^*$-central,  by Lemma~\ref{lem:supp-IP=subset}
	there exists $v\in U\cap B(x,\delta)$ such that $N(v, U\cap B(x,\delta))$
	contains an IP-subset of $N(x,X\setminus B(x,2\delta))$.
	For any $g\in N(v, U\cap B(x,\delta))$,
	$d(g v,x)<\delta$ and $d(g x, x)>2\delta$.
	Thus $d(g v,g x)>\delta$ for all $g$ in the IP-subset of $N(v, U\cap B(x,\delta))$, which contradicts to pairwise IP$^*$-equicontinuity of $x$.
\end{proof}

Combining Lemmas~\ref{lem:distal-system-IP-star-eq} and \ref{lem:IP-star-eq-point-distal}, we have the following result.

\begin{thm}\label{thm:ip-star-distal}
	An IP$^*$-central system is distal if and only if it is   pairwise IP$^*$-equicontinuous.
\end{thm}

Now we are ready to prove Theorem~\ref{thm:main-resut1}~(1).

\begin{proof}[Proof of Theorem~\ref{thm:main-resut1}~(1)]
Let $\mu$ be an invariant measure of $(X,G)$.
As $(X,G)$ is minimal, $\mu$ has full support.
By Lemma~\ref{lem:measure-IP-star-central}, $(X,G)$ is IP$^*$-central. Then the result follows from Theorem~\ref{thm:ip-star-distal}.
\end{proof}

\begin{rem}\label{rem:true-minimal-invariant}
 Similar to Proof of Theorem~\ref{thm:main-resut1}~$(1)$, it is easy to see that the conclusions of Lemma~\ref{lem:supp-IP=subset} and Lemma~\ref{lem:IP-star-eq-point-distal} are true for any minimal system which admits an invariant measure.
\end{rem}

\subsection{\texorpdfstring{Almost pairwise IP$^*$-equicontinuity}{Almost pairwise IP*-equicontinuity}}\label{subsection:almost-ip-star}
Let $(X,G)$ be a dynamical system. 
Denote by $\eq^{\ip^*}(X,G)$ the collection of all pairwise IP$^*$-equicontinuous points in $X$.
We say that a dynamical system $(X,G)$ is \emph{almost pairwise IP$^*$-equicontinuous}
if $\eq^{\ip^*}(X,G)$ is residual in $X$.

\begin{lem}\label{lem:IP-star-equi-points}
	Let $(X,G)$ be a dynamical system. Then $\eq^{\ip^*}(X,G)$ is a $G$-invariant $G_\delta$ subset of $X$.
\end{lem}

\begin{proof}
	For each $m\in\bbn$, denote by $\eq^{\ip^*}_m(X,G)$ the collection of all points $x$ in $X$ with the property that there exists a neighbourhood $U$ of $x$ such that for any $y,z\in U$,
	$\bigl \{g\in G\colon  d(g y,g z)<\frac{1}{m}\bigr\}$ is an IP$^*$-set.
	Clearly, $\{\eq^{\ip^*}_m(X,G)\}_{m=1}^\infty$ is a decreasing sequence of open subsets of $X$,
	and $\eq^{\ip^*}(X,G)=\bigcap_{m=1}^{\infty} \eq^{\ip^*}_m(X,G)$.
	Then $\eq^{\ip^*}(X,G)$ is a $G_\delta$ subset of $X$.
	
	Fix $h\in G$. For any $m\in\bbn$ there exists $n\in \bbn$ such that
	for any $u,v\in X$ with $d(u,v)<\frac{1}{n}$ one has
	$d(hu,hv)<\frac{1}{m}$.
	Assume that $x\in \eq^{\ip^*}_n(X,G)$, 
    that is there exists a neighbourhood $U$ of $x$ such that for any $y^{\prime},z^{\prime}\in U$,
	$\bigl \{g\in G \colon  d(g y^{\prime},g z^{\prime})<\frac{1}{n}\bigr\}$ is an IP$^*$-set.
	Let $V= hU$. Then $V$ is a neighborhood of $hx$.
	For any $y,z\in V$, $h^{-1}y,h^{-1}z\in U$.
	Then $\bigl \{g\in G\colon  d(g (h^{-1}y),g(h^{-1}z))<\frac{1}{n}\bigr\}$ is an IP$^*$-set.
	By the choice of $n$,
	$\bigl \{g \in G\colon  d(h g h^{-1}y,h g h^{-1}z)<\frac{1}{m}\bigr\}$ is also an IP$^*$-set.
    Note that $\bigl \{g \in G\colon  d( g y, gz)<\frac{1}{m}\bigr\} \supset h \bigl \{g \in G\colon  d(h g h^{-1}y,h g h^{-1}z)<\frac{1}{m}\bigr\}h^{-1}$, then
	$\bigl \{g \in G\colon  d( gy,gz)<\frac{1}{m}\bigr\}$ is also an IP$^*$-set.
	This implies that $h\eq^{\ip^*}_n(X,G)\subset  \eq^{\ip^*}_m(X,G)$, and then $h\eq^{\ip^*}(X,G)\subset \eq^{\ip^*}(X,G)$.
    By changing $h$ by $h^{-1}$, we have
    $h^{-1}(\eq^{\ip^*}(X,G))\subset \eq^{\ip^*}(X,G)$, 
    then we have $h(\eq^{\ip^*}(X,G))= \eq^{\ip^*}(X,G)$.
\end{proof}

\begin{prop}\label{prop:minimal-ip*-eq-points}
If $(X,G)$ is a minimal system, then either $\eq^{\ip^*}(X,G)$ is residual in $X$ or $\eq^{\ip^*}(X,G)$ is empty.
\end{prop}
\begin{proof}
Assume that $\eq^{\ip^*}(X,G)$ is not empty. Pick $x\in \eq^{\ip^*}(X,G)$.
By Lemma~\ref{lem:IP-star-equi-points} $\eq^{\ip^*}(X,G)$ is $G$-invariant, $Gx\subset \eq^{\ip^*}(X,G)$. As $(X,G)$ is minimal, $Gx$ is dense in $X$ and then $\eq^{\ip^*}(X,G)$
is also dense in $X$. By Lemma~\ref{lem:IP-star-equi-points} again, $\eq^{\ip^*}(X,G)$ is a $G_\delta$-subset of $X$. Hence $\eq^{\ip^*}(X,G)$ is residual in $X$.
\end{proof}

\begin{lem}\label{lem:distal-almost-1-1}
	Let $\pi\colon (X,G)\to (Y,G)$ be a factor map.
	If  $x\in X$ with $\pi^{-1}(\pi(x))=\{x\}$  and $(Y,G)$ is distal,
	then $x$ is pairwise IP$^*$-equicontinuous.
\end{lem}
\begin{proof}
	Fix $\eps>0$. As $\pi^{-1}(\pi(x))=\{x\}$, there exists a neighborhood $V$ of $\pi(x)$
	such that $\pi^{-1}(V)\subset B(x,\frac{\eps}{2})$.
	Pick a neighborhood $U$ of $x$ with $\pi(U)\subset V$.
	For any $u,v\in U$, $\pi(u),\pi(v)\in V$.
	As $\pi(u)$ and $\pi(v)$ are distal,
	$F:=\{g\in G\colon g \pi(u), g \pi(v)\in V\}$
	is an IP$^*$-set.
	For any $g\in F$, $g u,g v\in \pi^{-1}(V) \subset B(x,\frac{\eps}{2})$,
	and then $d(g u,g v)<\eps$.
	This implies that $x$ is pairwise IP$^*$-equicontinuous.
\end{proof}

Now we consider the opposite of pairwise IP$^*$-equicontinuity. 
A dynamical system $(X,G)$ is called \emph{pairwise IP-sensitive} 
if there exists a constant $\delta>0$ with the property that for each nonempty open subset $U$ of $X$,
there exist $x_1,x_2\in U$ such that
$\bigl\{g\in G \colon d(g x_1,g x_2)>\delta\bigr\}$
is an IP-set. We have the following dichotomy result for minimal systems. Note that the proof is different from the case of $\mathbb{N}$-actions which is proved in \cite[Theorem~3.10]{LY21}.

\begin{thm}\label{thm:dict-IP-star-eq}
    Every minimal system is either almost pairwise IP$^*$-equicontinuous or pairwise IP-sensitive.
\end{thm}
\begin{proof}
 Let $(X,G)$ be a minimal system. 
 First we assume that 	$\eq^{\ip^*}(X,G)\not=\emptyset$, then by Proposition~\ref{prop:minimal-ip*-eq-points} $\eq^{\ip^*}(X,G)$ is residual, that is, $(X,G)$ is almost pairwise IP$^*$-equicontinuous.
 
Now we assume that  $\eq^{\ip^*}(X,G)=\emptyset$, then every point in $X$ is not pairwise IP$^*$-equicontinuous.
For each $m\in\bbn$, denote by $A_m(X,G)$ the collection of all points $x$ in $X$ with the property that for each neighborhood $U$ of $x$,
there exist $x_1,x_2\in U$ such that
$\bigl\{g\in G \colon d(g x_1,g x_2)>\frac{1}{m}\bigr\}$
is an IP-set.  
It is clear that each $A_m(X,G)$ is a closed subset of $X$ and $\bigcup_{m\in\mathbb{N}} A_m(X,G)$ is the collection of all points in $X$ which are not  pairwise IP$^*$-equicontinuous. Then $\bigcup_{m\in\mathbb{N}} A_m(X,G)=X$.
By Baire category theorem, there exists $m_0\in\bbn$ such that the interior of $A_{m_0}(X,G)$ is not empty. Pick a nonempty open subset $V$ of $A_{m_0}(X,G)$.
As $(X,G)$ is minimal, there exists a finite subset $H$ of $G$ such that $\bigcup_{h\in H} hV=X$.
By the continuous of the action $G$, there exists a $\delta>0$ such that 
for any $u,v\in X$ with $d(u,v)>\frac{1}{m_0}$ one has $d(hu,hv)>\delta$ for each $h\in H$.

For every nonempty open subset $U$ of $X$, there exists  $h\in H$ such that $U\cap hV\neq\emptyset$. Then $V\cap h^{-1}U$ is a nonempty open subset of $A_{m_0}(X,G)$, and there exist $x_1',x_2' \in V\cap h^{-1}U$ such that
$\bigl\{g\in G \colon d(g x_1',g x_2')>\frac{1}{m_0}\bigr\}$
is an IP-set.  
Let $x_1=hx_1'$ and $x_2=hx_2'$. Then $x_1,x_2\in U$ and 
$\bigl\{g\in G \colon d(gh^{-1} x_1,gh^{-1} x_2)>\frac{1}{m_0}\bigr\}$
is an IP-set.  
By the choice of $\delta$, 
$\bigl\{g\in G \colon d(hgh^{-1} x_1, hgh^{-1} x_2)>\delta\bigr\}$
is an also IP-set. 
 Note that $ \bigl \{g \in G\colon d(g x_1, g x_2)>\delta\bigr\} \supset h\bigl\{g\in G \colon d(hgh^{-1} x_1, hgh^{-1} x_2)>\delta\bigr\} h^{-1}$, then
	$\bigl \{g \in G\colon d(g x_1, g x_2)>\delta \bigr\}$ is also an IP-set.
	This shows that $(X,G)$ is pairwise IP-sensitive. 
\end{proof}

\begin{prop}\label{prop:P-non-closed-IP-sen} 
Let $(X,G)$ be a minimal system which admits an invariant measure.
If the proximal relation $P(X,G)$ is not closed,  
then $(X,G)$ is pairwise IP-sensitive.
\end{prop}
\begin{proof}
	Since $P(X,G)$ is not closed,  there exists a distal pair $(y,z)$ and
	proximal pairs $(y_i,z_i)$ such that $(y_i,z_i)\rightarrow (y,z)$ as $i\to\infty$.
	Let $\delta:=\frac{1}{4}\inf_{g\in G}d(gy,gz)>0$.
	Fix a nonempty open subset $U$ of $X$.
	As $(X,G)$ is minimal, there exists $h\in G$ such that $hy\in U$.
	There exists $n\in \bbn$ such that $hy_n\in U\cap B(hy,\delta)$ and $d(hz_n,hz)<\delta$.
	Let $x_1=hy_n$ and $x_2=hz_n$.
	Then $x_1\in U$, $d(x_1,x_2)>\delta$ and $(x_1,x_2)$ is proximal.
	Choose nonempty open set $U_i$ containing $x_i$ such that $d(U_1,U_2)>\delta$ and $U_1\subset U$.
	As $x_2$ is a minimal point,
	we know that $N(x_1,U_2)$ is a central set and hence it contains an IP-set $FP(\{p_i\}_{i=1}^{\infty})$.
	By Lemma \ref{lem:supp-IP=subset} and Remark~\ref{rem:true-minimal-invariant}, there exists a sub-IP set $FP(\{q_j\}_{j=1}^{\infty})\subseteq FP(\{p_i\}_{i=1}^{\infty})$
	and $x_3\in U_1$ such that $gx_3\in  U_1$
	for each $g\in FP(\{q_j\}_{j=1}^{\infty})$. Then $d(gx_1, gx_3)>\delta$ for each $g\in FP(\{q_j\}_{j=1}^{\infty})$,
	this implies that $(X,G)$ is pairwise IP-sensitive.
\end{proof}

We have the following structure of almost pairwise IP$^*$-equicontinuous.
 
\begin{thm}\label{thm:stru-almost-IP*-eq}
	Let $(X,G)$ be a minimal system which admits an invariant measure.
	Then the following statements are equivalent:
	\begin{enumerate}
		\item $(X,G)$ is almost pairwise IP$^*$-equicontinuous;
         \item $(X,G)$ is point-distal and $P(X,G)$ is closed;
  		\item $\pi\colon (X,G)\to (X_d,G)$ is almost one-to-one, where $(X_d,G)$ is the maximal distal factor of $(X,G)$.
	\end{enumerate}
\end{thm}
\begin{proof} 
(1)$\Rightarrow$(2) 
As $(X,G)$ is almost pairwise IP$^*$-equicontinuous, the collection  $\eq^{\ip^*}(X,G)$ of all pairwise IP$^*$-equicontinuous points is residual in $X$. Since $(X,G)$ is a minimal system which admits an invariant measure, by Lemma~\ref{lem:IP-star-eq-point-distal} and Remark~\ref{rem:true-minimal-invariant}, every pairwise IP$^*$-equicontinuous point is distal. Then the collection of all distal points is residual in $X$, that is $(X,G)$ is point-distal. By Proposition~\ref{prop:P-non-closed-IP-sen}, we known that the proximal relation $P(X,G)$ is closed.

(2)$\Rightarrow$(3) Let
$\pi\colon (X,G)\to (X_d,G)$ be the factor map to its  maximal distal factor of $(X,G)$. As $P(X,G)$ is closed, by Lemma~\ref{lem:proximal-closed}, $P(X,G)$ is a $G$-invariant closed equivalence relation on $X$. Thus $R_\pi=P(X,G)$. Then for every point $x\in X$, $\pi^{-1}(\pi(x))=P(X,G)[x]$.
As $(X,G)$ is point-distal, the collection of all distal points is residual in $X$.
For every distal point $x\in X$, we claim that $P(X,G)[x]=\{x\}$.
Indeed, if $(x,y)$ is a proximal pair, as $x,y$ are minimal points then $ \overline{Gx}=\overline{Gy}$. In particular $y\in \overline{Gx}$. Then $x=y$ since $x$ is a distal point.
Therefore, for every distal point $x\in X$, $\pi^{-1}(\pi(x))$ is a singleton. This shows that $\pi$ is almost one-to-one.

(3)$\Rightarrow$(1) Let
$\pi\colon (X,G)\to (X_d,G)$ be the factor map to its  maximal distal factor of $(X,G)$ and $A=\{x\in X\colon \pi^{-1}(\pi(x))=\{x\}\}$.
As $\pi$ is almost one-to-one, $A$ is residual in $X$.
According to Lemma~\ref{lem:distal-almost-1-1}, every point in $A$ is pairwise IP$^*$-equicontinuous. Then $(X,G)$ is almost pairwise IP$^*$-equicontinuous. 
\end{proof}

The following result reveals that pairwise IP$^*$-equicontinuous points are the points with trivial section in the distal structure relation.
\begin{prop}
Let $(X,G)$ be a minimal system which admits an invariant measure. Then a point $x\in X$ is pairwise IP$^*$-equicontinuous if and only if $S_d(x)=\{x\}$, where $S_d$ is the distal structure relation of $(X,G)$.
\end{prop}
\begin{proof}
Let $\pi\colon (X,G)\to (X_d,G)$ be the factor map to its  maximal distal factor of $(X,G)$. Then $R_\pi=S_d$.
If $S_d(x)=\{x\}$, by Lemma~\ref{lem:distal-almost-1-1}, $x$ is pairwise IP$^*$-equicontinuous.
Now assume that $x$ is  pairwise IP$^*$-equicontinuous. 
By Proposition~\ref{prop:minimal-ip*-eq-points}, $(X,G)$ is almost pairwise IP$^*$-equicontinuous. 
By Lemma~\ref{lem:IP-star-eq-point-distal},  Remark~\ref{rem:true-minimal-invariant} and
the proof of Theorem~\ref{thm:stru-almost-IP*-eq}, we known that $S_d(x)=P(X,G)[x]=\{x\}$.
\end{proof}

\begin{rem}
Let $(X,G)$ be a minimal system and $\pi\colon (X,G)\to (X_d,G)$ the factor map to its the maximal distal factor of $(X,G)$. By \cite[VI (5.21)]{Vries1993}, the factor map $\pi$ has a largest almost  one-to-one factor $(Z,G)$. If in addition $(X,G)$ admits an invariant measure, then by Theorem~\ref{thm:stru-almost-IP*-eq}, $(Z,G)$ is also the maximal almost pairwise IP$^*$-equicontinuous factor of $(X,G)$, that is, every almost pairwise IP$^*$-equicontinuous factor of $(X,G)$ is a factor of $(Z,G)$.
\end{rem}

\section{\texorpdfstring{Pairwise central$^*$-equicontinuity and distality}{Pairwise central*-equicontinuity and distality}}

In this section we introduce a new kind of weak equicontinuity, named pairwise central$^*$-equicontinuity. 
We characterize distality via pairwise central$^*$-equicontinuity.  The structure theorem for (metric) minimal systems plays an important role in this section.

\subsection{\texorpdfstring{Pairwise central$^*$-equicontinuity}{Pairwise central*-equicontinuity}}

Let $(X,G)$ be a dynamical system.
A point $x$ in $X$ is called \emph{pairwise central$^*$-equicontinuous} 
if for any $\eps>0$ there exists a neighbourhood $U$ of $x$ such that 
for any $y,z\in U$, $\{g\in G\colon d(gy, gz)<\eps\}$ is a central$^*$-set. 
The dynamical system $(X,G)$ is called \emph{pairwise central$^*$-equicontinuous} if every point in $X$ is pairwise central$^*$-equicontinuous.

Since every IP$^*$-set is a central$^*$-set, every pairwise IP$^*$-equicontinuous point is pairwise central$^*$-equicontinuous and every pairwise IP$^*$-equicontinuous system is pairwise central$^*$-equicontinuous.

Similar to proof of Lemma~\ref{lem:IP*-eq-limit-distal}, we have the following result.
\begin{lem}\label{lem:central*-eq-limit-distal}
Let $(X,G)$ be a dynamical system. 
\begin{enumerate}
    \item 
If $x\in X$ is pairwise central$^*$-equicontinuous and a limit point of distal points, then $x$ is a distal point;
\item If $x\in X$ is pairwise central$^*$-equicontinuous and a limit point of central recurrent points, then $x$ is a central recurrent point.
\end{enumerate}
\end{lem}

Since we do not have the corresponding result of Lemma~\ref{lem:supp-IP=subset} for central sets, we should first study the conditions for the opposite of pairwise central$^*$-equicontinuity. 
A dynamical system $(X,G)$ is called \emph{pairwise central sensitive} if there exists a constant $\delta>0$ with the property that for each nonempty open subset $U$ of $X$,
there exist $x_1,x_2\in U$  such that $\bigl\{g\in G \colon d(g x_1,g x_2)>\delta\bigr\}$ is a central set.

\begin{prop}\label{prop:proximal-not-almost-1-1}
	Let $\pi:(X,G)\to (Y,G)$ be a factor map between minimal systems with $X\times X$ having a dense set of minimal points.
	If $\pi$ is proximal but not almost one-to-one, then $(X,G)$ is pairwise central sensitive.
\end{prop}
\begin{proof}
	As $\pi$ is not almost one-to-one, by~\cite[Lemma 3.3]{Li-Yang1}, $$\delta:=\frac{1}{4}\inf_{y\in Y}\diam(\pi^{-1}(y))>0.$$
    Fix a nonempty open subset $U$ of $X$.
	Since $(X,G)$ is minimal, $\pi$ is semi-open then the interior $\mathrm{Int}(\pi(U))$ of $U$ is not empty.
	Pick a point $y_0\in \mathrm{Int} (\pi(U))$. 
    Since $\diam(\pi^{-1}(y_0))\geq 4\delta>0$, we can find $u_1, u_2 \in \pi^{-1}(y_0)$ such that $d(u_1,u_2)>3\delta$.
	Let $W_i = B(u_i,\delta) \cap \pi^{-1}(\mathrm{Int}(\pi (U)))$, $i=1,2$.
	Then $W_1,W_2$ are nonempty open subsets of $X$ with $d(W_1,W_2)>\delta$.
	Since $(X\times X,G)$  has a  dense set of minimal points,
	choose a minimal point $(y_1,y_2)\in W_1\times W_2$ and points $x_i\in U$ with $\pi(x_i)=\pi(y_i)$ for $i=1,2$.
     As $\pi$ is proximal, it is easy to see that $\pi\times \pi\colon (X\times X,G)\to (Y\times Y,G)$ is also proximal (see e.g. \cite[Propsition~V(2.9)]{Vries1993}).
    Then $((x_1,x_2), (y_1,y_2))$ is proximal and
	$N((x_1, x_2), W_1\times W_2)$ is a central set, and 
	$\{g\in G\colon d(gx_1, gx_2) >\delta\}\supset N((x_1, x_2), W_1\times W_2)$. This shows that $(X,G)$ is pairwise central sensitive.
\end{proof}

We need the following structure theorem for (metric) minimal systems,
see \cite{EGS1975}.
\begin{thm}\label{thm:structure-thm-minimal-sys} 
For every minimal system $(X,G)$ there exists a countable ordinal $\eta$ 
and canonically determined minimal systems $(X_\nu,G)$, $(Y_\nu,G)$ and $(Z_\nu,G)$ with
$1\leq \nu\leq \eta$, 
and a commutative diagram 
\begin{equation*}
\small 
\xymatrix{
X \ar[d]^{\pi} \ar[dr]^{\sigma_1} &  &  X_1  \ar[d]^{\pi_1} \ar[ll]_{\theta^*_1}  & \dotsb  \ar[l] & X_\nu \ar[l] \ar[d]^{\pi_{\nu}} \ar[dr]^{\sigma_{\nu+1}} &  & X_{\nu+1} \ar[ll]_{\theta^*_{\nu+1}} \ar[d]^{\pi_{\nu+1}}&  \dotsb \ar[l] & X_\eta \ar[l] \ar[d]^{\pi_\eta} \\
\{\star\}  & Z_1 \ar[l]^{\rho_1} & Y_1 \ar[l]^{\theta_1}  & \ar[l] \dotsb  & Y_\nu \ar[l] & Z_{\nu+1} \ar[l]^{\rho_{\nu+1}} & Y_{\nu+1} \ar[l]^{\theta_{\nu+1}} & \dotsb \ar[l] & Y_\eta \ar[l]
}	
\end{equation*}
such that $\{\star\}$ is the singleton,  for each $\nu\leq \eta$, $\pi_\nu$ is RIC, $\rho_\nu$ is isometric, 
$\theta_{\nu}, \theta_{\nu}^*$ are proximal and $\pi_\eta$ is RIC and weakly mixing.
For a limit ordinal $\nu$, $X_\nu$, $Y_\nu$, $\pi_\nu$  etc.\@ 
are the inverse limits of $X_\iota$, $Y_\iota$, $\pi_\iota$  etc.\@  for $\iota<\nu$.
\end{thm}

\begin{rem}\label{rem:structure-min-sym}\leavevmode
\begin{enumerate}
    \item The definition of relatively incompressible (RIC) extension is more involved. We only need the fact that if $\pi\colon (X,G)\to (Y,G)$ is a RIC factor map between minimal system then $R_\pi$ has a dense set of minimal points.
    \item Since an inverse limit of proximal extension is also a proximal extension (see e.g. \cite[Corollary V(2.10)]{Vries1993}),
    the factor map $\theta\colon X_\eta\to X$ in Theorem~\ref{thm:structure-thm-minimal-sys} is proximal.
\end{enumerate}
\end{rem}

We say that a minimal system $(X,G)$ is a \emph{proximal-isometric system} (\emph{PI system} for short) if the factor map $\pi_\eta$ in the structure of $(X,G)$ is a homeomorphism. 
The following result is inspired by \cite[Proposition 5.5]{YY18}, but the proof here is more straightforward.

\begin{prop}\label{prop:non-PI-pairwise-cen-sen}
If a minimal system $(X,G)$ is not a PI system then it is pairwise central sensitive.
\end{prop}
\begin{proof}
By Theorem~\ref{thm:structure-thm-minimal-sys} and Remark~\ref{rem:structure-min-sym}, we have the following  diagram
\begin{equation*}
    \xymatrix{
X & X_\eta \ar[l]_\theta \ar[d]_{\pi_\eta} \\
& Y_\eta 
    }
\end{equation*}
where $\theta$ is proximal and $\pi_\eta$ is weak mixing and RIC.
Denote $R_\eta = \{(x,y)\in X_\eta \times X_\eta \colon \pi_\eta(x)=\pi_\eta(y)\}$.
Then $(R_\eta,G)$ is a transitive system with a dense set of minimal points.
As $(X,G)$ is not a PI system, $\pi_\eta$ is not a homeomorphism, then
$\{(x,x)\in X_\eta \times X_\eta\colon x\in X_\eta\}\subsetneq R_\eta$.
Let $R = \theta\times \theta (R_\eta)$. 
Then $(R,G)$ is a transitive system with a dense set of minimal points.
Since $\theta$ is proximal, $\Delta_X:=\{(x,x)\in X \times X \colon x\in X \}\subsetneq R$. 
Fix a minimal point $(z_1,z_2)\in R\setminus \Delta_X$.
Then there exists a $\delta>0$ such that for every point $(a,b)\in \overline{G(z_1,z_2)}$ one has $d(a,b)>3\delta$. 
Let $U$ be a nonempty open subset of $X$. Then $U\times U \cap R\neq\emptyset$ as $\Delta_X\subset R$.
Pick a transitive point $(x_1,x_2)$ in $U\times U \cap R$. 
By Theorem~\ref{thm:proximal}, there exists a minimal point $(y_1,y_2)\in \overline{G(z_1,z_2)}$
such that $((x_1,x_2),(y_1,y_2))$ is proximal. 
Then $N((x_1,x_2),B(y_1,\delta)\times B(y_2,\delta))$ is a central set, and  $\bigl\{g\in G \colon d(g x_1,g x_2)>\delta\bigr\}\supset N((x_1,x_2),B(y_1,\delta)\times B(y_2,\delta))$.
This shows that $(X,G)$ is  pairwise central sensitive.
\end{proof}

We need the following characterization of PI systems, see e.g. \cite[Page 570]{Vries1993}.
\begin{thm}\label{thm:PI-sys-AI-sys}
A minimal system $(X,G)$ is a PI system if and only if every transitive subsystem of $(X\times X,G)$ with a dense set of minimal points is minimal.
\end{thm}

Now we are ready to prove Theorem~\ref{thm:main-resut1}(2).

\begin{proof}[Proof of Theorem~\ref{thm:main-resut1}(2)]
We only need to prove that for a minimal system $(X,G)$ with 
a dense set of minimal points in the product system $(X\times X,G)$,
if $(X,G)$ is pairwise central$^*$-equicontinuous then it is distal.
Let $(X,G)$ be a pairwise central$^*$-equicontinuous system. 
First by Proposition~\ref{prop:non-PI-pairwise-cen-sen}, $(X,G)$ is a PI system. 
As the intersection of two central$^*$-sets is still a central$^*$-set,
the product system $(X\times X,G)$ is pairwise central$^*$-equicontinuous.
As $(X\times X,G)$ has a dense set of minimal points, it is clear that every minimal point is central recurrent, then by Lemma~\ref{lem:central*-eq-limit-distal}, every point in $(X\times X,G)$ is central recurrent.
Fix any point $(x,y)\in X\times X$. 
Then $(\overline{G(x,y)},G)$ is a transitive system.
For every nonempty open subset $U$ of $\overline{G(x,y)}$, 
pick a nonempty open subset $V$ of $U$ with $\overline{V}\subset U$ and a point $(a,b)\in V$. As $(a,b)$ is central recurrent,
$N((a,b),V)$ is a central set. 
By Proposition~\ref{prop:central-chara}, there exists a minimal point in $\overline{V}$. Then $(\overline{G(x,y)},G)$ has a dense set of minimal points. Now by Theorem~\ref{thm:PI-sys-AI-sys}, $\overline{G(x,y)}$ is minimal. In particular $(x,y)$ is a minimal point. According to Lemma~\ref{lem:distal-prod-pointwise-minimal}, $(X,G)$ is distal.
\end{proof}

\section{\texorpdfstring{Pairwise FIP$^*$-equicontinuity and systems of order $\infty$}{Pairwise FIP*-equicontinuity and systems of order infty}}

In this section we focus on systems of order $\infty$. We aim to characterize minimal systems of order $\infty$ via a weak equicontinuity: pairwise FIP$^*$-equicontinuity, and study the local property of this kind of weak equicontinuity.
In Proposition~\ref{prop:supp-FIP=subset} we characterize FIP-set by a return time set, which may be of interest independently.

\subsection{\texorpdfstring{Systems of order $\infty$}{Systems of order infty}}
Let $(X,G)$ be a minimal system and let $k\ge 1$ be an integer. A pair $(x, y) \in X\times X$
is said to be \emph{regionally proximal of order $k$}
if for any $\delta
> 0$, there exist $x', y'\in X$ and a
sequence $\{p_i\}_{i=1}^k$ in $G$
such that $d(x, x') < \delta$, $d(y, y') <\delta$, and 
\[
d(gx', gy') < \delta\ \text{for
any}\ g\in FP(\{p_i\}_{i=1}^k).
\]
The \emph{regionally proximal relation of order $k$}, denoted by $\RP^{[k]}(X)$, is the collection of all regionally proximal pairs of order $k$.
It is clear that
\[
 P(X,G)\subset \dotsb \subset\RP^{[k+1]}(X,G)\subset
\RP^{[k]}(X,G)\subset \dotsb \subset \RP^{[2]}(X,G)\subset \RP^{[1]}(X,G).
\]

When the action group is abelian, we have the following results on the regionally proximal relation of order $k$ for minimal system.
\begin{thm}[{\cite[Theorem 7.7]{SY2012}}] \label{thm:RPd-icer}
If $(X,G)$ is minimal system with $G$ being abelian, then for every $k\in\mathbb{N}$, $\RP^{[k]}(X,G)$ is a $G$-invariant closed equivalence relation.
\end{thm}

\begin{thm}[{\cite[Proposition 8.15]{HSY2016}}] \label{thm:RPd-fip-d+1}
Let $(X,G)$ be a minimal system with $G$ being abelian and $k\in\mathbb{N}$.
Then a pair $(x, y) \in \RP^{[k]}(X,G)$ if and only if for any neighborhood $U$ of $y$ there exists a
sequence $\{p_i\}_{i=1}^{k+1}$ in $G$
such that $FP(\{p_i\}_{i=1}^{k+1})\subset N(x,U)$.
\end{thm}

Let $\RP^{[\infty]}(X,G)=\underset{k=1}{\overset{\infty}{\bigcap}}\RP^{[k]}(X,G)$. Following from Theorems~\ref{thm:RPd-icer} and~\ref{thm:RPd-fip-d+1}, one has

\begin{thm}\label{thm:PR-infty}
    Let $(X,G)$ be a minimal system with $G$ being abelian. Then 
    \begin{enumerate}
        \item $\RP^{[\infty]}(X,G)$ is a $G$-invariant closed equivalence relation;
        \item a pair $(x, y) \in \RP^{[\infty]}(X,G)$ if and only if for any neighborhood $U$ of $y$, $N(x,U)$ is an FIP-set.
    \end{enumerate}
\end{thm}
Following \cite{DSMSY13} and \cite{GGY18}, we say that a dynamical system $(X,G)$ is a \emph{system of order $\infty$} if $\RP^{[\infty]}(X,G)$ is trivial, i.e. it coincides with the diagonal of $X\times X$.
If  $(X,G)$ is a  minimal system with $G$ being abelian, by Theorem~\ref{thm:PR-infty} $X/\RP^{[{\infty}]}(X,G)$ is the maximal factor of order $\infty$.

A point $x\in X$ is called \emph{$\infty$-step almost automorphic} if $\RP^{[\infty]}(X,G)[x]=\{x\}$.
It is clear that $(X,G)$ is a system of order $\infty$ if and only if
every point in $(X,G)$ is $\infty$-step almost automorphic.
The following characterization of $\infty$-step almost automorphic points was proved in \cite[Theorem 8.1.7]{HSY2016}, see also \cite[Theorem~0.2]{BL18}.

\begin{thm}
\label{thm:AA-fip-star-rec}
	Let $(X,G)$ be a minimal system with $G$ being abelian.
	Then a point $x\in X$ is an $\infty$-step
	almost automorphic point if and only if it is FIP$^*$-recurrent.
\end{thm}

\subsection{\texorpdfstring{Pairwise FIP$^*$-equicontinuity}{Pairwise FIP*-equicontinuity}}
Let $(X,G)$ be a dynamical system.
A point $x$ in $X$ is called \emph{pairwise FIP$^*$-equicontinuous} if for any $\eps>0$ there exists a neighbourhood $U$ of $x$ such that for any $y,z\in U$, $\{g\in G\colon d(gy, gz)<\eps\}$ is an FIP$^*$-set. 
Denote by $\eq^{FIP^*}(X,G)$ the collection of all pairwise FIP$^*$-equicontinuous points in $X$.
A dynamical system $(X,G)$ is called \emph{pairwise FIP$^*$-equicontinuous} if $\eq^{FIP^*}(X,G)=X$.

We will prove Theorem~\ref{thm:main-resut2} in this subsection, first we need the following result which is implied in~\cite{Gillis}, see \cite[Proposition 5.8]{HLY11}  for a proof of this version.

\begin{prop}\label{prop:measure-skill}
Let $(X,\mathcal{B},\mu)$ be a probability space and
$\{E_i\}_{i=1}^\infty$ be a sequence in $\mathcal{B}$ with $\mu(E_i)\geq a>0$ for some constant $a$ and any $i\in\mathbb{N}$.
Then for any $k\geq 1$ and $\eps>0$ there is $N=N(a,k,\eps)\in\mathbb{N}$ such that for any strictly increase sequence $\{s_i\}_{i=1}^n$ in $\mathbb{N}$ with $n\geq N$, there exist
$1\leq t_1<t_2<\dotsb<t_k\leq n$ with 
\[
\mu\bigl(E_{s_{t_1}}\cap E_{s_{t_2}}\cap\dotsb\cap E_{s_{t_k}}\bigr)\geq a^k-\eps.
\]
\end{prop}

\begin{prop}\label{prop:supp-FIP=subset}
	If a dynamical system  $(X,G)$ admits an invariant measure with full support,
	then for any FIP-set $F$ and nonempty open subset $U$ of $X$, there exists
	an FIP-subset $F^{\prime}$ of $F$ and a point $z\in U$ such that $F^{\prime}\subset N(z,U)$.
\end{prop}
\begin{proof}
Let $\mu$ be a $G$-invariant measure with full support.
We first prove the following claim.

\smallskip 
\noindent \textbf{Claim:} For every $n\in\bbn$ and Borel subset $E$ of $X$ with $\mu(E)\geq a>0$ for some constant $a$.
There exists $k=k(a,n)\in\mathbb{N}$ such that for any sequence $\{p_i\}_{i=1}^k$ in $G$ there exists  
a sequence $\{q_i\}_{i=1}^n$ in $G$ such that 
$FP(\{q_i\}_{i=1}^n)\subset FP(\{p_i\}_{i=1}^{k})$ 
and 
\[
\mu\biggl(E\cap \bigcap_{g\in FP(\{q_i\}_{i=1}^n)}g^{-1}E\biggr)\geq c_n,
\]
where $c_1=\frac{1}{2}a^2$ and $c_{i+1} =\frac{1}{2} c_i^2$ for $i\in\mathbb{N}$.
\begin{proof}[Proof of Claim]
Let $E$ be a Borel subset $E$ of $X$ with $\mu(E)\geq a$.
We prove the Claim by induction on $n$. 

For $n=1$, let $k=k(a,1)=N(a,2,\frac{1}{2}a^2)$ as in Proposition~\ref{prop:measure-skill}.
For any sequence $\{p_i\}_{i=1}^k$ in  $G$,
as $\mu$ is $G$-invariant, $\mu((\prod_{i=1}^j p_i)^{-1}E)=\mu(E)\geq a$ for all $j=1,\dotsc,k$.
By Proposition~\ref{prop:measure-skill} there exists $1\leq j_1<j_2\leq k$ such that 
$\mu((\prod_{i=1}^{j_1} p_i)^{-1} (E) \cap (\prod_{i=1}^{j_2} p_i)^{-1}(E))\geq \frac{1}{2}a^2$.
Let $g= \prod_{i=j_1+1}^{j_2} p_i$. Then
\[
\mu(E\cap g^{-1}E)
=\mu\biggl(\biggl(\prod_{i=1}^{j_1} p_i\biggr)^{-1} (E) \cap \biggl(\prod_{i=1}^{j_2} p_i\biggr)^{-1}(E)\biggr)\geq \frac{1}{2}a^2. 
\]
This shows that the result holds for $n=1$.

Assume that the result holds for $n\leq m$. 
For $n=m+1$, let $k=k(a,m+1)=k(a,m)+k(c_{m},1)$.
For any sequence $\{p_i\}_{i=1}^k$ in $G$, there exists  
a sequence $\{q_i\}_{i=1}^m$ in $G$ such that 
$FP(\{q_i\}_{i=1}^m)\subset FP(\{p_i\}_{i=1}^{k(a,m)})$ 
and 
\[
\mu\biggl(E\cap \bigcap_{g\in FP(\{q_i\}_{i=1}^m)}g^{-1}E\biggr)\geq c_m.
\]
Let $V= E\cap \bigcap_{g\in FP(\{q_i\}_{i=1}^m)}g^{-1}E$.
For the sequence $\{p_i\}_{i=k(a,m)+1}^k$, 
there exists $q_{m+1}\in FP(\{p_i\}_{i=k(a,m)+1}^k)$ such that 
\[
\mu(V\cap q_{m+1}^{-1}V)\geq  c_{m+1}.
\]
Then 
$FP(\{q_i\}_{i=1}^{m+1})\subset FP(\{p_i\}_{i=1}^{k})$ 
and 
\[
\mu\biggl(E\cap \bigcap_{g\in FP(\{q_i\}_{i=1}^{m+1})}g^{-1}E\biggr)
=\mu(V\cap q_{m+1}^{-1}V)
\geq  c_{m+1}.
\]
This ends the proof of the claim.
\end{proof}
Fix  an FIP-set $F$ and a nonempty open subset $U$ of $X$.
For every $k\in\bbn$ there exists  
a sequence $\{p_i^{(k)}\}_{i=1}^k$ in $G$ such that 
$FP(\{p_i^{(k)}\}_{i=1}^k)\subset F$.
Take a nonempty open subset $V_1$ of $X$ with $\overline{V_1}\subset U$.
Let $a_1=\mu(V_1)$. Then $a_1>0$. 
Let $k_1=k(a_1,1)$ as in the Claim. 
Then there exists $q_1^{(1)}\in G$ such that $q_1^{(1)}\in FP(\{p_i^{(k_1)}\}_{i=1}^{k_1})$
and $\mu(V_1\cap (q_1^{(1)})^{-1}V_1)>0$. 
Assume that $a_m$, $k_m=k(a_m,m)$, $V_m$, $\{q_i^{(m)}\}_{i=1}^m$ has been chosen for $m\leq n$ such that $FP(\{q_i^{(m)}\}_{i=1}^m)\subset FP(\{p_i^{(k_m)}\}_{i=1}^{k_m})$ and 
\[
\mu\biggl(V_m \cap \bigcap_{g\in FP(\{q_i^{(m)}\}_{i=1}^{m})}g^{-1}V_m\biggr)>0.
\]
Pick a nonempty open subset $V_{n+1}$ of  $X$ with 
$\overline{V_{n+1}}\subset V_n \cap \bigcap_{g\in FP(\{q_i^{(n)}\}_{i=1}^{n})}g^{-1}V_n$.
Let $a_{n+1}=\mu(V_{n+1})$. Then $a_{n+1}>0$. 
Let $k_{n+1}=k(a_{n+1},n+1)$ as in the Claim. 
Then there exists a sequence $\{q_i^{(n+1)}\}_{i=1}^{n+1}$ in $G$ such that
$FP(\{q_i^{(n+1)}\}_{i=1}^{n+1})\subset FP(\{p_i^{(k_{n+1})}\}_{i=1}^{k_{n+1}})$ and 
\[
\mu\biggl(V_{n+1} \cap \bigcap_{g\in FP(\{q_i^{(n+1)}\}_{i=1}^{n+1})}g^{-1}V_{n+1}\biggr)>0.
\]
By induction, we get a sequence of nonempty open subsets $\{V_k\}_{k=1}^\infty$ and an FIP set $F'=\bigcup_{k=1}^\infty FP(\{q_i^{(k)}\}_{i=1}^{k})$. It is clear that $F'\subset F$. Pick a point $z\in \bigcap_{k=1}^\infty \overline{V_k}$. Then $z\in U$ and $F'\subset N(z,U)$.
\end{proof}

\begin{thm}\label{thm:min-FIP-eq-point-aa}
 Let $(X,G)$ be a minimal system with $G$ being abelian. 
 Then a point $x\in X$ is pairwise FIP$^*$-equicontinuous if and only if it is $\infty$-step almost automorphic.
\end{thm}
\begin{proof}
($\Rightarrow$) 
Let $x\in X$ be a pairwise FIP$^*$-equicontinuous point. As the action group $G$ is abelian, $(X,G)$ admits an invariant measure $\mu$. Moreover, $(X,G)$ is minimal, $\mu$ has full support. Similar to the proof of Lemma~\ref{lem:IP-star-eq-point-distal}, using Theorem~\ref{thm:AA-fip-star-rec} and Proposition~\ref{prop:supp-FIP=subset}, one has that $x$ is $\infty$-step almost automorphic.

($\Leftarrow$) Let  $\pi\colon (X,G)\to (X_\infty,G)$ be the factor map to the maximal factor of order $\infty$.
By Theorem~\ref{thm:PR-infty}(1),
$\RP^{[\infty]}(X,G)$ is a $G$-invariant closed equivalence relation, $X_\infty=X/\RP^{[{\infty}]}(X,G)$ and 
$R_\pi=\RP^{[\infty]}(X,G)$. For any almost automorphic point $x$, by the definition  $R_\pi[x]=\RP^{[\infty]}(X,G)[x]=\{x\}$. Similar to the proof of Lemma~\ref{lem:distal-almost-1-1}, using Theorem~\ref{thm:AA-fip-star-rec}, one has that $x$ is pairwise FIP$^*$-equicontinuous.
\end{proof}

Now Theorem~\ref{thm:main-resut2} is an immediate consequence of Theorem~\ref{thm:min-FIP-eq-point-aa}.

\subsection{\texorpdfstring{Almost FIP$^*$-equicontinuity}{Almost FIP*-equicontinuity}}

Similar to Lemma~\ref{lem:IP-star-equi-points} and Proposition~\ref{prop:minimal-ip*-eq-points}, we have the following characterization of $\eq^{FIP^*}(X,G)$.
\begin{lem}
Let $(X,G)$ be a dynamical system. Then $\eq^{FIP^*}(X,G)$ is a $G$-invariant $G_\delta$ subset of $X$.
\end{lem}

\begin{prop}\label{prop:fip-empty-residual}
If $(X,G)$ be a minimal system. Then either $\eq^{FIP^*}(X,G)$ is residual or $\eq^{FIP^*}(X,G)$ is empty.
\end{prop}

Similar to the case of pairwise IP$^*$-equicontinuity, we consider the local property and the opposite of pairwise FIP$^*$-equicontinuity. A dynamical system $(X,G)$ is called
\emph{almost pairwise FIP$^*$-equicontinuous}
if $\eq^{FIP^*}(X,G)$ is residual in $X$, 
and \emph{pairwise FIP-sensitive} 
if there exists a constant $\delta>0$ with the property that for each nonempty open subset $U$ of $X$,
there exist $x_1,x_2\in U$ such that
$\bigl\{g\in G \colon \rho(g x_1,\allowbreak g x_2)>\delta\bigr\}$
is an FIP-set. 

Similar to Theorem~\ref{thm:dict-IP-star-eq}, we have the following dichotomy result  for minimal systems.

\begin{thm}
 Let $(X,G)$ be a minimal system.
 Then $(X,G)$ is either almost pairwise FIP$^*$-equicontinuous or pairwise FIP-sensitive.
\end{thm}

Recall that a minimal system $(X,G)$ is called \emph{$\infty$-step almost automorphic} if it has some $\infty$-step almost automorphic point. 
By Theorem~\ref{thm:min-FIP-eq-point-aa} and Proposition~\ref{prop:fip-empty-residual}, we have the following corollary.

\begin{cor}\label{cor:stru-almost-FIP*-eq}
 Let $(X,G)$ be a minimal system with $G$ being abelian. 
 Then $(X,G)$ is almost pairwise FIP$^*$-equicontinuous if and only if it is   $\infty$-step almost automorphic.
\end{cor}

\subsection*{Acknowledgments.}
The authors would like to thank Prof. Weisheng Wu  and Dr. Jiahao Qiu for helpful suggestions.
We also express many thanks to the anonymous referee, whose comments have substantially improved this paper.

\end{document}